%% file: planarapprox.tex
\documentclass[reqno,11pt]{amsart}
\usepackage{a4wide,color,eucal,enumerate,mathrsfs}
\usepackage[normalem]{ulem}
\usepackage{amsmath,amssymb,epsfig,amsthm} 
\usepackage[latin1]{inputenc}
\usepackage{psfrag}
\usepackage{epsfig}
\usepackage{color} 
\usepackage{graphicx}
\usepackage{transparent}

\def\az{\alpha}

\def\dist{{\mathop\mathrm{\,dist\,}}}
\def\loc{{\mathop\mathrm{\,loc\,}}}

\def\ez{\epsilon}

\def\boz{{\Omega}}

\def\wz{\widetilde}

\def\bint{{\ifinner\rlap{\bf\kern.35em--}
\int\else\rlap{\bf\kern.45em--}\int\fi}\ignorespaces}

\def\bbint{{\ifinner\rlap{\bf\kern.35em--}
\hspace{0.078cm}\int\else\rlap{\bf\kern.45em--}\int\fi}\ignorespaces}

\def\esup{\mathop\mathrm{\,esssup\,}}

\def\diam{{\mathop\mathrm{\,diam\,}}}

\newtheorem{thm}{Theorem}[section]
\newtheorem{lem}[thm]{Lemma}

\newtheorem{cor}[thm]{Corollary}
\newtheorem{defn}[thm]{Definition}
\numberwithin{equation}{section}

\theoremstyle{remark}

\def\bint{{\ifinner\rlap{\bf\kern.35em--}
\int\else\rlap{\bf\kern.45em--}\int\fi}\ignorespaces}

\begin{document}

\title[A density problem for Sobolev spaces on planar domains]
{A density problem for Sobolev spaces on planar domains}

\author{Pekka Koskela}
\author{Yi Ru-Ya Zhang}

\address{Department of Mathematics and Statistics \\
         P.O. Box 35 (MaD) \\
         FI-40014 University of Jyv\"as\-kyl\"a \\
         Finland}
\email{pekka.j.koskela@jyu.fi} 
\email{yi.y.zhang@jyu.fi}

\thanks{This work was supported by the Academy of Finland via the Centre of Excellence in Analysis and Dynamics Research (Grant no. 271983)}
\date{\today}



\begin{abstract}
We prove that for a bounded simply connected domain $\Omega\subset \mathbb R^2$, the Sobolev space $W^{1,\,\infty}(\Omega)$ is dense in $W^{1,\,p}(\Omega)$ for any $1\le p<\infty$. Moreover, we show that if $\Omega$ is Jordan, then $C^{\infty}(\mathbb R^2)$ is dense in $W^{1,\,p}(\Omega)$  for $1\le p<\infty$. 
\end{abstract}


\maketitle

\tableofcontents

\section{Introduction}

Let $\boz \subset\mathbb R^2$ be a domain. 
We define the first order Sobolev space $W^{1,p}(\boz),$ $1 \le p \le \infty,$ 
as the set
$$\left\{u \in L^p(\boz)\mid \nabla  u \in L^p(\boz;\,\mathbb R^2) \right\}. $$
Here  $\nabla u=(\frac{\partial u}{\partial x_1},\,  
\frac{\partial u}{\partial x_2})$ is the weak (or distributional) gradient
of a locally integrable function $u.$
We equip $W^{1,p}(\boz)$ with the non-homogeneous norm: 
$$\|u\|^p_{W^{1,p}(\boz)} = \int_{\boz} | u(x)|^p \, dx + \int_{\boz} |\nabla u(x)|^p \, dx$$
for $1 \le p < \infty$, and
$$\|u\|_{W^{1,\infty}(\boz)} =\esup_{x\in \boz} |u(x)|+\esup_{x\in \boz}|\nabla u(x)|. $$

For $1 \le p < \infty$, it is well-known that smooth functions are dense in $W^{1,p}(\Omega)$ for any domain $\Omega\subset \mathbb R^2$. 
Consequently, if $\Omega$ is a $ W^{1,p}$-extension domain, that is, there 
exists an extension operator $E\colon W^{1,\,p}(\Omega)\to 
W^{1,\,p}(\mathbb R^2):$ 
\begin{equation*}
Eu(x)=u(x) \text{ for } x\in \Omega \ \ \ \mbox{and} \ \ \ \ Eu\in W^{1,\,p}(\mathbb R^2),
\end{equation*}
then we can use 
global smooth functions to approximate functions in $ W^{1,p}(\Omega)$ 
with respect to $ W^{1,p}(\Omega)$-norm. 
Indeed one extends $u\in W^{1,\,p}(\Omega)$ to $Eu\in W^{1,\,p}(\mathbb R^2)$, picks a sequence $v_j\in C^{\infty}(\mathbb R^2)$ approximating $Eu$ in $W^{1,\,p}(\mathbb R^2)$-norm and then restricts these $v_j$ to 
$\overline{\Omega}$. 
Notice that Lipschitz domains are extension domains. 

If one only wishes to approximate by functions that are smooth up to the 
boundary, then the Lipschitz condition can be relaxed. Indeed,
if $\Omega$ satisfies a cone condition or the weaker segment condition, 
then $C^{\infty}(\overline\Omega)$ is dense in $W^{1,\,p}(\Omega)$.
However, it is easy to construct domains $\Omega$ for which 
$C^{\infty}(\overline\Omega)$ 
fails to be dense. For example, take $\Omega$ to be a slit disk:
the unit disk minus a radius. For all this see e.g. \cite{AF2003}.

A very different sufficient condition for the density of global smooth 
functions was given by J.L. Lewis in \cite{L1985}. He proved that 
$C^{\infty}(\mathbb R^2)$ is dense in $ W^{1,p}(\Omega)$ for every  
$1<p<\infty$ when $\Omega$ is a Jordan domain: the bounded component of 
$\mathbb R^2\setminus \gamma$, 
where $\gamma$ is a Jordan curve. Lewis's approximation procedure is based on 
extending the restriction of the function in question, from a suitable level
set, smoothly along the normal vector field of a fixed weak solution of the 
$q-$Laplace equation
$$\nabla\cdot(|\nabla u|^{q-2}\nabla u)=0$$
in $\Omega'$, where $\frac 1 q+\frac 1 p=1$ and $\Omega\subset\subset\Omega'$. 
In order to obtain the required estimates, he uses properties of solutions of 
these equations. This technique leaves the case $p=1$ open.

Subsequently, W. Smith, A. Stanoyevitch and D.A. Steganga showed in 
\cite{SSS1994} that domains which satisfy their interior segment property 
allow approximation of functions in $W^{1,p}(\Omega)$ for $1\le p<\infty$
by bounded smooth 
functions with bounded derivatives and with global smooth functions if 
the boundary of $\Omega$ satisfies a suitable additional 
exterior density condition.
Their interior segment property is weaker than the usual segment property
that actually implies that the boundary is locally the graph of a continuous
function. In \cite{SSS1994} it was also inquired if the measure density 
together with lack of two-sided boundary points would suffice for the density 
of $C^{\infty}(\overline\Omega),$ but C.J. Bishop \cite{B1996} gave a
counterexample to this statement.

More recently, A. Giacomini and P. Trebeschi established in \cite{GT2007} 
density results that especially yield the density of $W^{1,\,2}(\Omega)$ in
$W^{1,\,p}(\Omega)$ for all $1\le p <2$ when $\Omega$ is bounded and simply
connected. They use the Helmholtz decomposition 
of $L^2(\Omega,\,\mathbb R^2)$ to characterize the orthonormal subspaces of 
certain Sobolev spaces. Thus only the density of $W^{1,\,2}(\Omega)$ can be
obtained by this technique.

Based on the results above, it is natural to inquire if $W^{1,\,q}(\Omega)$ is
always dense in $W^{1,\,p}(\Omega)$ for some $q>p$ when $\Omega$ is bounded
and simply connected and if even global smooth functions are dense in 
$W^{1,1}(\Omega)$ when $\Omega$ is Jordan. Our first result gives an even
stronger conclusion for the the first problem.

\begin{thm}\label{mainthm 1}
If $\Omega\subset \mathbb R^2$ is a bounded simply connected domain, 
then $W^{1,\,\infty}(\Omega)$ is dense in $W^{1,\,p}(\Omega)$
for 
any $1\le p< \infty.$  
\end{thm}

There are plenty of bounded simply connected non-Jordan domains that fail the
interior segment condition and hence Theorem~\ref{mainthm 1} is not
covered by the results discussed above. Our proof of Theorem~\ref{mainthm 1} 
is rather flexible. Especially, it allows us to solve the second problem posed
above and to give a new proof
for the aforementioned density result  \cite[Theorem 1]{L1985} by J.L. Lewis; 
other consequences of our approach will be recorded in a subsequent
paper.

\begin{cor}\label{global density}
If $\Omega$ is a planar Jordan domain, then 
$C^{\infty}(\mathbb R^2)$ is dense in $W^{1,\,p}(\Omega)$
for any $1\le p< \infty$. 
\end{cor}

Theorem~\ref{mainthm 1} and Corollary~\ref{global density} also give 
consequences for $BV(\Omega)$, 
the collection of functions in $L^1(\Omega)$ with bounded variation.
Indeed, given $u\in BV(\Omega)$ one always has a sequence of functions 
$u_j\in W^{1,1}(\Omega)$ that converges to $u$ in $L^1(\Omega)$ and
so that the $BV$-energy of $u,$ $||Du||(\Omega),$ satisfies
$$||Du||(\Omega)=\lim_j||\nabla u||_{L^1(\Omega)}.$$
Based on Theorem~\ref{mainthm 1} and Corollary~\ref{global density}, we may
further assume that $u_j\in W^{1,\,\infty}(\Omega)$ when $\Omega$ is bounded
and simply connected and even that each 
$u_j$ is the restriction
of a global smooth function when $\Omega$ is Jordan. For the theory
of $BV$-functions we refer the reader to \cite{AFP}.
 
The research of this paper has been partially motivated by
our attempts to give geometric characterizations for bounded simply connected
$W^{1,p}$-extension domains. Indeed, our solution for the case $1<p<2$ 
in \cite{kry1} uses Lewis's result (Corollary~\ref{global density} for $p>1$). 
For the case
$p=1$ we need both Theorem~\ref{mainthm 1} and Corollary~\ref{global density},
see \cite{kry2}.

The paper is organized as follows. Section 2 contains some 
preliminaries. We give a decomposition of a bounded 
simply connected planar 
domain $\Omega$ and the corresponding partition of unity in Section 3. 
Section 4 contains the proofs of Theorem~\ref{mainthm 1} and 
Corollary \ref{global density}. 

The notation in this paper is quite standard. For example, $C(\cdot)$ refers
to a constant that may depend on the given parameters. As usual, the value
$C(\cdot)$ may
vary between appearences, even within a chain of inequalities.
By $a \sim b$ we mean that $b/C \le a \le Cb$ for some constant 
$C \ge 2$.
If we need to make the dependence of this costant on the parameters 
$(\cdot)$ explicit, we write $a {\sim}_{(\cdot)} b$. 
Also $a \lesssim b$ means $a\le C b$ with $C\ge 1$, and 
$a\gtrsim b$ has the analogous meaning. 
The Euclidean distance between the sets $A,\,B \subset \mathbb R^n$ is 
denoted by $\dist(A,\,B)$. 
We denote by $\ell(\gamma)$  the length of a curve $\gamma$. 
The Euclidean disk centered at $x$ and with radius $r$ is referred to by 
$B(x,\,r)$, and $S^{1}(x,\,r)$ is the circle of radius $r,$ centered at $x.$  
For a set $A\subset \mathbb R^n$, we refer to its interior by $A^o,$ 
to the boundary by $\partial A,$  and to the closure by $\overline A.$ 
As usual, $A\subset\subset B$ means that $A$ is compactly contained in $B$.

\section{Preliminaries}

In this section, we introduce some necessary definitions and facts. 
To begin with, we recall the definition of a Whitney-type set.

\begin{defn}\label{whitney-type set}
 A connected set $A \subset \boz\subset \mathbb R^2$ is called a 
$\lambda$-Whitney-type set in $\Omega$ 
with some 
constant $\lambda\ge 1$ if the following holds. 

(i) There exists a disk of radius $\frac {1}{\lambda}\diam(A)$ contained 
in $A$;

(ii) $ \frac {1}{\lambda} \diam(A)\le \dist(A,\,\partial\Omega)\le {\lambda } \diam(A)$. 
\end{defn}

We define the {\it{inner distance with respect to $\Omega$}} between 
$x,\,y\in\Omega$ by
$$\dist_{\Omega}(x,\,y)=\inf_{\gamma\subset \Omega} \ell(\gamma),$$
where the infimum runs over all curves joining $x$ and $y$ in $\Omega.$
The inner diameter $\diam_{\Omega}(E)$ of a set $E\subset \Omega$ is then defined in
the usual way.

Let us recall some facts from complex analysis. 
First of all, recall that conformal maps preserve conformal 
capacities. More precisely, given a pair $E,\,F \subset \Omega\subset \mathbb R^2$ of continua, define
the {\it conformal capacity between $E$ and $F$ in $\Omega$} as
 $${\rm Cap}(E,\,F,\,\Omega)=\inf\{\|u\|^2_{ W^{1,\,2} (\Omega)}\mid  u\in\Delta(E,\,F )\},$$
 where $\Delta(E,\,F )$ denotes the class of all  
functions $u\in W^{1,\,2}_{\loc}(\Omega),$ continuous in $E\cup F\cup \Omega,$
 such that $0\le u\le 1$ on $\Omega$, $u=1$ on $E$, and $u=0$ on $F$. 
By definition, it is clear that if $\wz E,\,\wz F\subset\wz \Omega,$  
$E\subset \wz E$, $F\subset \wz F$ and $\Omega\subset \wz \Omega$,
  then $${\rm Cap}(E,\,F,\,\boz)\le {\rm Cap}(\wz E,\,\wz F,\, \wz \Omega).$$

Furthermore if $E,\,F\subset \mathbb D$ are two continua, then
\begin{equation}\label{condition of lower bound}
\frac{\min\{\diam(E),\,\diam(F)\}}{\dist(E,\,F)}\ge \delta>0 \quad 
\Longrightarrow \quad {\rm Cap}(E,\,F,\,\mathbb D)\ge C(\delta)>0.
\end{equation}
Moreover, when $\overline{B(x,\,r)}\subset B(x,\,R)\subset\subset \Omega$, we 
have 
\begin{equation}\label{capacity of balls}
{\rm Cap}(\overline{B(x,\,r)},\,S^1(x,\,R),\,\Omega)\sim \log\left(\frac R r\right)^{-1}. 
\end{equation}
 See \cite{V1971} for more details. Actually, \cite{V1971}
states these results for ``modulus'', but ``modulus'' is equivalent to
conformal capacity (see e.g. \cite[ Proposition 10.2, Page 54]{R1993}).

We need the following lemma.

\begin{lem}{\label{inner capacity}}
Let $E,\,F\subset \Omega$ be a pair of continua. 
Then if ${\rm Cap}(E,\,F,\, \Omega)\ge c_0$, we have
$${\min\{\diam_{\Omega}(E),\,\diam_{\Omega}(F)\}}\gtrsim {\dist_{\Omega}(E,\,F)}, $$
where the constant only depends on $c_0$. 
\end{lem}
\begin{proof}
We may assume that $\diam_{\Omega}(E)\le \diam_{\Omega}(F)$ and 
$2\diam_{\Omega}(E) \le \dist_{\Omega}(E,\,F)$. 
Let $z\in E$, and $\frac{\dist_{\Omega}(E,\,F)}{\diam_{\Omega}(E)}=\delta$. 
We define
$$ f(x)=
 \begin{cases}
   1, & \text{if } \dist_{\Omega}(x,\,z)\le \diam_{\Omega}(E) \\
  0, & \text{if}  \dist_{\Omega}(x,\,z)\ge \dist_{\Omega}(E,\,F)\\
      \frac {\log (\dist_{\Omega}(E,\,F))-\log (\dist_{\Omega}(x,\,z))}{\log (\dist_{\Omega}(E,\,F))-\log (\diam_{\Omega}(E))}, &\text{otherwise }
 \end{cases}.$$

Then a direct calculation via a dyadic annular decomposition with respect
to the inner distance gives
$$c_0\le \int_{\Omega} |\nabla f|^2\, dx\lesssim {\log (\delta)}^{-1}. $$
Hence $\delta\lesssim e^{c_0}$, which means that $\dist_{\Omega}(E,\,F) \lesssim \diam_{\Omega}(E)$. 
\end{proof}

Recall that hyperbolic geodesics in $\mathbb D$ are arcs of 
(generalized) circles that intersect the 
unit circle orthogonally. Moreover, both the hyperbolic metric and hyperbolic
geodesics are preserved under conformal maps; 
see \cite[Page 37]{AIM2009} for instance.
We refer to the hyperbolic distance between a pair of points $x,y$ in a 
simply connected planar domain by 
$\dist_{h}(x,\,y).$

The following lemma states a distortion property of conformal maps. 

\begin{lem}[\cite{AIM2009}, Theorem 2.10.8]\label{linear map}
Suppose $\varphi$ is conformal in the unit disk $\mathbb D$ and $z,\,w\in \mathbb D$. Then
$$\exp{(-3\dist_h(z,\,w))}|\varphi'(w)|\le |\varphi'(z)| \le\exp{(3\dist_h(z,\,w))}|\varphi'(w)|. $$
\end{lem}

Given a $\lambda$-Whitney type set $A\subset \mathbb D,$ one has that 
$\dist_h(z,\,w)\le C(\lambda)$ for all $z,w\in A,$ 
Hence $|\varphi'(w)|\sim |\varphi'(z)|$ with a constant only depending on 
$\lambda.$

By Lemma~\ref{linear map}, condition \eqref{condition of lower bound} and the 
capacity estimate 
\eqref{capacity of balls}, one can verify the following well-known result. 

\begin{lem}\label{whitney preserving}
Suppose $\varphi \colon \Omega \to \Omega'$ is conformal, where $\Omega=\mathbb D$ or 
$\Omega'=\mathbb D,$
and $Q\subset \Omega$ is a $\lambda_1$-Whitney-type set. Then 
$\varphi(Q)\subset \Omega'$ is a $\lambda_2$-Whitney-type set with $ \lambda_2=\lambda_2(\lambda_1)$.  \end{lem}

In the sequel, we often omit the constant $\lambda$ when we are dealing with 
a fixed $\lambda.$ 

Hyperbolic geodesics have the following important property, often
called the Gehring-Hayman inequality. 

\begin{lem}[\cite{GH1962}]\label{GH thm}
Let $\varphi \colon \mathbb D \to \Omega$ be a conformal map. 
Then for any two points $x,\,y\in \mathbb D$, 
denoting the corresponding hyperbolic geodesic in $\Omega$ by 
$\gamma_{x,\,y}$, and by $\omega_{x,\,y}$ any 
Jordan curve connecting $x$ and $y$ in $\mathbb D$, we have
$$\ell(\varphi(\gamma_{x,\,y}))\le C \ell(\varphi(\omega_{x,\,y})),$$
where $C$ is an absolute constant. 
\end{lem}

Finally, let us recall that bounded smooth functions are 
dense in $W^{1,\,p}(\Omega)$. 

\begin{lem}\label{bounded}
For any $1\le p<\infty$, it holds that 
$L^{\infty}(\Omega)\cap C^{\infty}(\Omega)$ 
is dense in $W^{1,\,p}(\Omega)$ for any domain $\Omega\subset \mathbb R^2$. 
\end{lem}
\begin{proof}
Fix $v\in W^{1,\,p}(\Omega)$. 
Let $$v_m=\max\{\min\{v(x),\,m\},\,-m\}.$$ One can easily check by
the absolute continuity of integral that this
sequence converges to $v$ in the Sobolev norm.  
The claim follows by a standard partition of unity and mollification 
procedure applied to the functions $v_m.$
 
\end{proof}

\section{Decomposition and partition of unity}\label{Decomposition and partition of unity}
\begin{figure}
 \centering
 \def\svgwidth{300pt}
 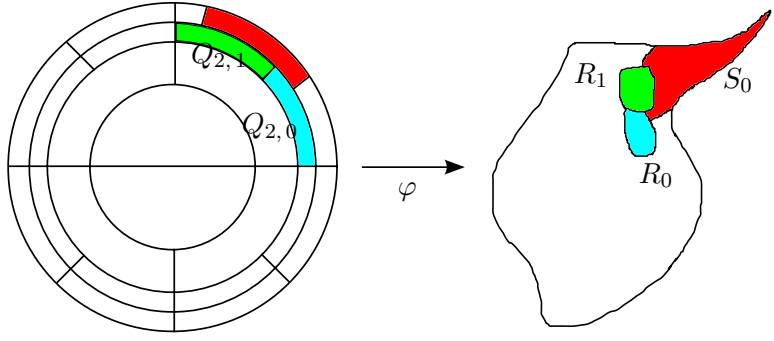
 \caption{The sets $Q_{2,\,0},\,Q_{2,\,1}$ and the corresponding images under $\varphi$, namely $R_0,\,R_1$ and the set $S_0$.}
 \label{fig:decomposition}
\end{figure}
\subsection{Decomposition of the core of $\Omega$}
Fix a bounded simply connected domain $\Omega\subset\mathbb R^2$, and 
consider a conformal map $\varphi\colon \mathbb D \to \Omega$.
For $l\ge 1$, let
$$A_l=\overline{B}(0,\,1-2^{-l-1})\setminus B(0,\,1-2^{-l}). $$
We define the radial ray $r_{\theta}$ as the line segment between the origin and the point $e^{i\theta}$. 
For each $l\in \mathbb N$, $0 \le j \le 2^{l+1}-1$ and 
$\theta_{l,\,j}= j 2^{-l}\pi$,  the collection of radial rays 
$r_{\theta_{l,\,j}}$ cut $A_l$ into $2^{l+1}$ sets $Q_{l,\,j}$ labeled counterclockwise respect to $j$ starting from the positive real axis. Moreover, define
$A_0=B(0,\,\frac 1 2)$, $Q_{0,\,0}=\{(x_1,\,x_2)\in A_0 \mid x_2\ge 0 \},$
and let $Q_{0,\,1}=\{(x_1,\,x_2)\in A_0 \mid x_2\le 0 \}. $
By abuse of notation, we sometimes refer also to the closures of the sets $Q_{l,\,j}$ 
by $Q_{l,\,j}.$ 
Notice that all these sets are of Whitney-type.

For $m\ge 2$, set 
$$\Omega_m=\varphi(\overline{B}(0,\,1-2^{-m-1})), $$
and 
$$D_m=\varphi(A_m)\subset \Omega_m.$$
For $0\le j\le 2^{m+1}-1$, by Lemma~\ref{whitney preserving}, the induced set $R_{j}=\varphi( Q_{m,\,j})\subset \Omega$ is also a Whitney type set for $\Omega$. These sets form a decomposition of $D_m$. Apparently the set $R_j$ depends on $m$, but for notational convenience we suppress this. 

\subsection{Decomposition of the boundary layer of $\Omega$}
Now let us decompose $J_m=\Omega\setminus\Omega_m$. Our aim is to decompose $J_m$ into connected sets such that, for each of them the length of its boundary inside $\Omega$ is controlled, and the distance between any two sets is relatively far if they have no intersection. 

To be specific, for each $0\le j\le 2^{m+1}-1$, define $\beta_{j}$ to be the shorter arc of
$$S^1(0,\,1-2^{-m-1})\setminus (r_{ (2j+1) 2^{-m-1}\pi}\cup r_{ (j+1) 2^{-m}\pi}),$$ 
so that
\begin{equation}\label{condition 1}
 2^{-m-2}\pi\le \dist(\beta_j,\,\beta_{j+1})\le 2^{-m+1}\pi.
\end{equation}

We claim that there exists a hyperbolic geodesic $\gamma^n_{j}$ connecting $\varphi(\beta_{j})\subset \partial R_j$ and $\varphi(\delta^n_{j})$ such that $\ell(\gamma^n_{j})\lesssim \diam(R_{j})$, where each $\delta^n_{j}$ is the shorter arc of
$$S^1(0,\,1-2^{-m-n})\setminus (r_{ (2j+1) 2^{-m-1}\pi}\cup r_{ (j+1) 2^{-m}\pi})$$ 
for $n\ge 3$. Observe that $\diam(\varphi(\beta_{j}))\sim \diam(R_{j})\sim
\diam_{\Omega}(R_j)$ by 
Lemma~\ref{linear map}.

Notice that ${\rm Cap}(\beta_{j},\,\delta^n_{j},\,\mathbb D)$ is bounded away 
from zero by an absolute constant according to 
\eqref{condition of lower bound}, and hence 
$${\rm Cap}(\varphi(\beta_{j}),\,\varphi(\delta^n_{j}),\,\Omega)\gtrsim 1. $$ 
By Lemma~\ref{inner capacity}, we conclude that 
$\dist_{\Omega}(\varphi(\beta_{j}),\,\varphi(\delta^n_{j}))\lesssim \diam_{\Omega}(R_{j})$. The existence of a suitable $\gamma_j^n$ follows by Lemma~\ref{GH thm}. 

Parameterize each $\gamma^n_j$ by arclength. 
Notice that the lengths of $\gamma^n_j$ are uniformly bounded from above by a multiplicative
constant times $\diam(R_j)$.
 Letting $n\to \infty$, by Arzel\'a-Ascoli lemma, we obtain a curve 
$\gamma_j$ connecting $R_j$ and the boundary $\partial \Omega$ with 
$\ell(\gamma_j) \lesssim \diam(R_{j}).$ 

Moreover by Lemma~\ref{linear map}, for any $0\le j\le 2^{m+1}-1$, $\diam_{\Omega}(R_j)\sim\diam_{\Omega}(R_{j+1}),$ where 
we define $R_{2^{m+1}}=R_0$. 
Thus by the triangle inequality and Lemma~\ref{linear map}
\begin{equation}\label{oneside}
\dist_{\Omega}(\gamma^n_j,\,\gamma^n_{j+1})\lesssim \diam(R_j), 
\end{equation}

Addtionally, we claim that 
\begin{equation}\label{anotherside}
\dist_{\Omega}(\gamma^n_j,\,\gamma^n_{j+1})\gtrsim \diam(R_j). 
\end{equation}
Indeed, first of all, by construction and \eqref{condition 1} we know that 
$\gamma^n_j \cap \gamma^n_{j+1} =\emptyset$.  Consider a curve 
$\az\subset \Omega$ of length at most 
$2\dist_{\Omega}(\gamma^n_j,\,\gamma^n_{j+1})$, joining 
$\gamma^n_j,\,\gamma^n_{j+1}$ in $\Omega.$

If $16\dist(\varphi^{-1}(\az),\, Q_{m,\,j})\le \diam(Q_{m,\,j})$, then there is 
a subarc $\az'\subset \az$ such that 
$\ell(\varphi^{-1}(\az'))\sim \diam(Q_{m,\,j})$ and 
$\dist(\varphi^{-1}(\az'),\, Q_{m,\,j})\le \frac 1 8\diam(Q_{m,\,j})$. 
Then by Lemma~\ref{linear map} and \eqref{condition 1} one concludes that 
$$\diam(R_j) \lesssim\diam(\az')\le \ell(\az')\le \ell(\az)\sim \dist_{\Omega}(\gamma^n_j,\,\gamma^n_{j+1}). $$

For the other case where $16\dist(\varphi^{-1}(\az),\, Q_{m,\,j})\ge \diam(Q_{m,\,j})$, observe that 
$${\rm Cap}(\varphi^{-1}(\az),\, Q_{m,\,j},\,\Omega)\gtrsim 1,$$ 
and by Lemma~\ref{linear map} $\dist(\az,\,R_j)\gtrsim \diam(R_j)$. Hence by 
Lemma~\ref{inner capacity} we conclude that
$$\diam(R_j) \lesssim \diam(\az)\le \ell(\az )\sim \dist_{\Omega}(\gamma^n_j,\,\gamma^n_{j+1}).$$
Consequently we obtain the claim. Combining \eqref{oneside} and \eqref{anotherside} results in
\begin{equation*}
\dist_{\Omega}(\gamma^n_j,\,\gamma^n_{j+1})\sim \diam(R_j), 
\end{equation*}
and finally in
\begin{equation}\label{curve separation}
\dist_{\Omega}(\gamma_j,\,\gamma_{j+1})\sim \diam(R_j), 
\end{equation}
by letting $n\to \infty$. 

Denote by $S_j$ the relatively closed subset of
$\Omega$ enclosed by $\partial \Omega$, $\partial \Omega_m$, $\gamma_j$ and 
$\gamma_{j+1}$. Then $J_m=\cup_j S_j$ and $|S_i\cap S_j|= 0$ for any $i\neq j$, where $|A|$ refers to the Lebesgue measure of a set $A$.  Thus the sets $S_j$, modulo sets of measure zero, give us a decomposition of $J_m$. 

Furthermore, based on \eqref{curve separation}, we claim that 
\begin{equation}\label{set separation}
\dist_{\Omega}(S_i,\,S_j)\gtrsim \max\{\diam(R_i),\,\diam(R_j)\} \text{ if } R_i\cap R_j=\emptyset, 
\end{equation}
with a constant independent of $\Omega$ and $m$. 
Indeed any curve $\gamma\subset \Omega$ joining $S_i$ and $S_j$ must pass 
through  the neighbors of $S_i$, namely $R_i\cup R_{i+1}\cup R_{i+2}\cup S_{i-1}\cup S_{i+1}$. A similar conclusion holds also for $S_j$ and its neighbors. Then the desired claim is given by \eqref{curve separation}, the definition of
the sets $R_i,\, R_{i-1}$ and $R_{i+1}$ and  Lemma~\ref{linear map}.

\subsection{A partition of unity associated to the decomposition}

Next we construct a partition of unity related to the decomposition above. 
Recall that
$$\Omega=\Omega_m\cup J_m,$$ $\Omega_m=\Omega_{m-1}\cup D_m,$
and for $D_m\subset \Omega_m$ and $J_m$, we have 
$D_m=\cup_j R_j$ and $J_m=\cup_j S_j$ respectively. 

For $\Omega_m$, we define a Lipschitz function $\psi$ in $\Omega$ such that 
$\psi$ is compactly supported in $\Omega_m$, $0\le \psi \le1$, $\psi (x)=1$ 
if $x\in \Omega_{m-1}$, and $|\nabla \psi(x)|\lesssim (\diam (R_j))^{-1}$ 
if $x\in R_j$, with a constant independent of $m,j.$
This function can be given via the distance function by letting 
$$\psi(x)=\min\left\{1, \frac  {c_1\dist(x,\,J_m)} {\dist(x,\,\partial
\Omega)}\right\},$$   
where the value of $c_1$ will be fixed momentarily.

Indeed, $\psi$ is Lipschitz and,
by Leibniz's rule, for $x\in R_j$
$$|\nabla \psi|\lesssim \frac 1 {\dist(x,\,\partial \Omega)}+ 
\frac {\dist(x,\,J_m)}  {(\dist(x,\,\partial \Omega))^2}\lesssim \frac 1 {\diam(R_j)},$$
where we applied the fact that $\diam(R_j)\sim\dist(R_j,\,\partial \Omega)$ 
for any $j$. Since $\psi$ vanishes in $J_m,$ we are left to obtain the 
correct boundary value on $\Omega_{m-1}.$ For this, notice that each
$x\in \partial \Omega_{m-1}$ belongs to $R_j$ for some $j.$ Since
$R_j$ is a Whitney-type set, 
$$\dist(x,\partial \Omega) \lesssim \diam (R_j).$$
On the other hand, Lemma~\ref{linear map} guarantees that
$$\diam(R_j)\lesssim \dist(x,J_m).$$
Hence there is a constant $C_1,$ independent of $m,j$ so that
$\psi=1$ on $\Omega_{m-1}$ provided $c_1\ge C_1.$

For each $S_j$, we choose a locally Lipschitz continuous function $\phi_j$ 
defined in $\Omega$ such that the support of $\phi_j$ is relatively closed in 
$\Omega$ and contained in $c_2 S_j$, $0\le \phi_j\le1$, $\phi_j(x)=1$ if 
$x\in S_j$, and $|\nabla \phi_j|\lesssim (\diam(R_j))^{-1}$. 
Here the set $c_2 S_j$ is defined as
$$c_2 S_j=\left\{x\in\Omega \mid \dist_{\Omega}(x,\,S_j)\le  {(c_2-1)\diam(R_j)}\right\}, $$
for some constant $c_2>1$ to be determined later.
Indeed, we can simply set 
$$\phi_{j}(x)=\max\{1-2 [(c_2-1)\diam(R_j)]^{-1} \dist_{\boz}(x,\,  S_j),\, 0\}$$
for $x\in \Omega.$

Let us now choose $c_2$ small enough, so that \eqref{set separation} and 
Lemma~\ref{whitney preserving} guarantee that
\begin{equation}\label{support separation}
c_2S_i\cap c_2S_j=\emptyset \ \ \text{ if} \ \ S_i\cap S_j=\emptyset,  \quad c_2S_i\cap R_j=\emptyset \ \ \text{ if} \ 
\ R_{i+1}\cap R_j=\emptyset 
\end{equation}
and
\begin{equation}\label{support separation2}
c_2S_i\cap \Omega_{m-1}=\emptyset
\end{equation}
for each $i.$

Towards obtaining a partition of unity, we wish now to choose $c_1$ large enough
so that
$\psi(x)+\phi_j(x)\gtrsim 1$ for each $x\in R_j.$ Notice that Lemma~\ref{linear map} gives us a constant $C_2$ and the fact that $R_j$ is a Whitney type
set a constant $C_3$ so that
$$\dist(x,\,S_{j-1}\cup S_j)\le C_2\dist(x,\,J_m)$$ 
and
$$\dist(x,\,\partial\Omega)\le C_3\diam(R_j)$$
when $x\in R_j.$ Choosing 
$$c_1=2\max\left\{C_1,\frac {C_2C_3}{c_2-1}\right\}$$
does the job; then $\phi_{j-1}(x)+\phi_j(x)\ge \frac 1 4$ if $x\in R_j$ and 
$\psi(x)\le \frac 1 2.$

We conclude that $\Phi(x):=\psi(x)+\sum_j \phi_j(x)\ge 1/4$ for each 
$x\in \Omega$ and
hence 
we obtain the desired partition of unity by dividing $\varphi$ and each 
$\phi_j$ by $\Phi$ in $\Omega.$ By our construction,
the new functions that
we still denote $\psi$ and $\phi_i$ for convenience satisfy the
same gradient bounds as the original ones, up to a multiplicative constant.

\section{Proof of Theorem~\ref{mainthm 1} and Corollary~\ref{global density}}

The proof of Theorem~\ref{mainthm 1} is based on approximating our given
function $u$ via a weighted sum of the functions in the partition of unity
from the previous section. Towards this end, for $m$ to be fixed later and
the associated indices $j,$
define 
$$a_j=\bint_{\varphi^{-1}(R_j)} u\circ \varphi\, dx.$$
Then $a_j$ is the average of $u\circ \varphi$ over $Q_{m,j}=\varphi^{-1}(R_j).$
Recall here our notation from Section 2.

We need the following technical result. 

\begin{lem}\label{Poincare inequality}
For $u\in W^{1,\,p}(\Omega)$ and $R_i,\, R_{i+1}\subset \Omega$ defined in 
Section~\ref{Decomposition and partition of unity}, we have
$$|a_j-a_{j+1}|^p\lesssim (\diam(R_j))^{p-2}\int_{R_j\cup R_{j+1}} |\nabla u|^p\, dx$$
and
$$\int_{R_j}|u-a_j|^p\, dx\lesssim (\diam(R_j))^{p}\int_{R_j} |\nabla u|^p\, dx,$$
where the constant only depends on $p$. 
\end{lem}
\begin{proof}
First of all by Lemma~\ref{linear map}, we know that 
$u\circ \varphi\in W^{1,\,p}_{\loc}(\mathbb D)$. We apply the usual 
Poincar\'e inequality on the nice domain $Q_{m,j}=\varphi^{-1}(R_j)$
to get
$$\int_{Q_{m,j}}|u\circ \varphi-a_j|^p\, dx\lesssim \diam(Q_{m,j})^p
\int_{Q_{m,j}}|\nabla (u\circ \varphi)|^p\, dx.$$
Notice that $J_{\varphi}(z)=|\varphi'(z)|^2$ by conformality of $\varphi.$
Hence our second estimate follows via a change of variable by using 
chain rule and 
Lemma~\ref{linear map}, according to which $\varphi'$ is essentially constant
on $Q_{m,i}.$ 

The first inequality follows analogously, using now the Poincar\'e inequality
over $Q_{m,j}\cup Q_{m,j+1}$ and by adding and subtracting the average over
$Q_{m,j}\cup Q_{m,j+1}.$
\end{proof}

\begin{proof}[Proof of Theorem~\ref{mainthm 1}]
Fix $\ez>0$. Also fix $u\in W^{1,\,p}(\Omega)$ for given $1\le p<\infty$. 
We may assume that $u$ is smooth and bounded because of Lemma~\ref{bounded}. 
We may also require that $\|u\|_{L^{\infty}(\Omega)}=1$. 

For $m\in\mathbb N$ large enough
\begin{equation}\label{requirement}
\|u\|^p_{W^{1,\,p}(J_m\cup D_m)}\le \ez \ \text{ and  }\  |J_m\cup D_m|\le \ez. 
\end{equation}
Notice that $u|_{\Omega_m}\in W^{1,\,\infty}(\Omega_m)$ since $\Omega_m$ is compact and $u$ is smooth. 

We define a function $u_m$ on $\Omega$ by setting
$$u_m(x)=u(x)\psi(x)+\sum_{j} a_j \phi_j(x),$$
where $\psi(x)$ and $\phi_j(x)$ are the corresponding functions in the 
partition of unity from the previous section
and $a_j$ is as in the beginning of this section.

It is obvious that $u_m\in W^{1,\,\infty}(\Omega)$ is locally Lipschitz by our 
construction, since we only have finitely many $R_j$ and the definition of our 
partition of unity gives the right estimates on the derivatives of of the
functions in our partition of unity. 
Moreover we have $\|u_m\|_{L^{\infty}(\Omega)}\le 1$ since $\|u\|_{L^{\infty}(\Omega)}=1$, and hence $$\|u_m\|_{L^{p}(J_m\cup D_m)}\le \ez.$$
Consequently, since $c_2S_j\cap \Omega_{m-1}=\emptyset$ for any $j$, we only 
need to check that
\begin{equation*}\label{aim for k}
\int_{J_m\cup D_m}|\nabla u_m|^p\, dx \lesssim \ez.
\end{equation*}
This actually follows via the Poincar\'e inequality, 
Lemma~\ref{Poincare inequality}.

Indeed for any $R_i\subset D_m$ with the associated constant $a_i$, Lemma~\ref{Poincare inequality} and \eqref{support separation} give
\begin{align*}
&\int_{R_i} |\nabla u_m|^p \, dx \lesssim \int_{R_i} |\nabla (u_m-a_i)|^p \, dx \\
\lesssim& \int_{R_i} |\nabla [ (u(x)-a_i)\psi(x)]|^p \, dx +\sum_{\substack{S_j\subset J_m\\ c_2S_j \cap R_i \neq\emptyset}} \int_{R_i} |\nabla [(a_j-a_i) \phi_j(x) ]|^p \, dx \\
\lesssim& \int_{R_i} |\nabla u|^p + |u(x)-a_i|^p \diam(R_i)^{-p}\, dx +
\sum_{\substack{S_j\subset J_m\\ c_2S_j \cap R_i \neq\emptyset}} 
|a_i-a_j|^p \diam(R_i)^{2- p} \\
\lesssim& \int_{R_i} |\nabla u|^p \, dx+ \sum_{\substack{S_j\subset J_m\\ c_2 S_j \cap R_i \neq\emptyset}}  \int_{R_i\cup R_j}|\nabla u|^p \,dx , 
\end{align*}
where $R_j$ and $R_i$ are the corresponding Whitney-type sets contained in $D_m$ for $S_j$ and $S_i,$ respectively. 

Next for each $S_i$, by letting its associated constant to be $a_i$, by Lemma~\ref{Poincare inequality}, \eqref{support separation} and the definition
of $\phi_j,$ we get
\begin{align*}
&  \int_{S_i} |\nabla u_m|^p \, dx \lesssim  \int_{S_i} |\nabla (u_m-a_i)|^p\, dx \\
& \lesssim \sum_{\substack{S_j\subset J_m\\ c_2S_i \cap c_2S_j \neq\emptyset}} \int_{S_i} |\nabla [ (a_i-a_j) \phi_j(x) ]|^p \, dx \lesssim \sum_{\substack{S_j\subset J_m\\ c_2S_i \cap c_2S_j \neq\emptyset}}  |a_j-a_i|^p \diam(R_i)^{2-p} \\
&\lesssim \sum_{\substack{S_j\subset J_m\\ c_2S_i \cap c_2S_j \neq\emptyset}} \int_{R_i\cup R_j} |\nabla u|^p  \, dx, 
\end{align*}
where $R_j$ and $R_i$ are still the corresponding Whitney-type sets contained in $D_m$ for $S_j$ and $S_i,$ respectively. 

Since all the sets $R_j$ and $c_2S_j$ have uniformly finitely 
many overlaps, the desired estimate follows by summing over $i.$
\end{proof}

Let $X$ and $Y$ be two non-empty subsets of $\mathbb R^n$. Define the {\it Hausdorff distance} $\dist_{\mathrm H}(X,\, Y)$ between them as
$$\dist_{\mathrm H}(X\,,Y) = \max\{\,\sup_{x \in X} \inf_{y \in Y} d(x,\,y),\, \sup_{y \in Y} \inf_{x \in X} d(x,\,y)\}. $$
We are ready to prove Corollary~\ref{global density}. 

\begin{proof}[Proof of Corollary~\ref{global density}]
For a given Jordan domain $\Omega\subset \mathbb R^2$ we can construct a 
sequence 
of Lipschitz domains $\{G_s\}_{s=1}^{\infty}$ approaching it in Hausdorff 
distance such that $\Omega\subset \subset G_{s+1}\subset \subset G_{s}$ for 
each $s\in\mathbb N$. For example, define $G_s$ by subtracting from 
$\mathbb R^2$ all the closed Whitney squares of the complementary domain of 
$\Omega$ whose sidelength is larger than $2^{-s}$.

Let us recall the proof of Theorem~\ref{mainthm 1}. For a function 
$u\in L^{\infty}(\Omega)\cap C^{\infty}(\Omega)\cap W^{1,\,p}(\Omega)$, we first 
restricted it on $\Omega_m$ so that \eqref{requirement} is satisfied, where 
the corresponding sets $J_m$ and $D_m$ are defined in 
Section~\ref{Decomposition and partition of unity}. Then we extended the 
restricted function $u_m$ to each set $S_j$ as the integral average of $u$ on 
the corresponding set $R_j$. Next we "glued" these pieces together by our
partition of unity, such that the non-zero gradient of $u-u_m$ can only 
appear in the neighborhoods (with respect to the topology of $\Omega$) of 
$\partial \Omega_m$  and of the curves $\gamma_j$. We remind that 
\eqref{set separation} was crucial here.

Now let us return to  Corollary~\ref{global density}. 
Fix $u\in C^{\infty}(\Omega)\cap W^{1,\,p}(\Omega)\cap L^{\infty}(\Omega)$ 
and $\ez>0$. When $m$ is large enough, we still truncate $u$ on $\Omega_m$ so 
that \eqref{requirement} holds.

First observe that, when $s$ is large enough, a Whitney-type set contained in 
$\Omega_m$ is still a Whitney-type set in $G_s$ up to a multiplicative 
constant, as the domains $G_s$ converge to $\Omega$ in Hausdorff distance. 
Especially, all the sets $R_j$ are still of Whitney-type.

We furthermore require that $\dist_{\mathrm H}(G_s,\,\Omega)$ is much smaller 
than the smallest value among $\{\diam(\gamma_j)\}_j$; notice that this is a 
finite collection. Then, if we extend the end point $z_j\in \partial \Omega$ of 
each $\gamma_j$ to one of the nearest points on $\partial G_s$, formulas 
similar to \eqref{curve separation} and \eqref{set separation} still hold for 
the new curves and sets. Consequently a decomposition of $G_s$ with a 
corresponding partition of unity can also be constructed via the essence of 
Section~\ref{Decomposition and partition of unity}.

Thus an argument similar to the proof of Theorem~\ref{mainthm 1} can 
be employed  for $G_s$. Since $G_s$ is a Lipschitz domain, we can extend 
functions in 
$W^{1,\,\infty}(G_s)$ to global Lipschitz functions. Hence we get a sequence of 
global Lipschitz functions approximating $u$ in $W^{1,\,p}(\Omega)$-norm. 
Applying suitable mollifiers and following a standard diagonal argument, 
we obtain a sequence of global smooth functions as desired. 
\end{proof}

\end{document}

%% file: drawing.pdf_tex
\begingroup%
  \makeatletter%
  \providecommand\color[2][]{%
    \errmessage{(Inkscape) Color is used for the text in Inkscape, but the package 'color.sty' is not loaded}%
    \renewcommand\color[2][]{}%
  }%
  \providecommand\transparent[1]{%
    \errmessage{(Inkscape) Transparency is used (non-zero) for the text in Inkscape, but the package 'transparent.sty' is not loaded}%
    \renewcommand\transparent[1]{}%
  }%
  \providecommand\rotatebox[2]{#2}%
  \ifx\svgwidth\undefined%
    \setlength{\unitlength}{372.44375bp}%
    \ifx\svgscale\undefined%
      \relax%
    \else%
      \setlength{\unitlength}{\unitlength * \real{\svgscale}}%
    \fi%
  \else%
    \setlength{\unitlength}{\svgwidth}%
  \fi%
  \global\let\svgwidth\undefined%
  \global\let\svgscale\undefined%
  \makeatother%
  \begin{picture}(1,0.41670722)%
    \put(0,0){\includegraphics[width=\unitlength]{drawing.pdf}}%
    \put(0.23168662,0.33960482){\color[rgb]{0,0,0}\makebox(0,0)[lb]{\smash{$Q_{2,\,1}$}}}%
    \put(0.29570927,0.25097699){\color[rgb]{0,0,0}\makebox(0,0)[lb]{\smash{$Q_{2,\,0}$}}}%
    \put(0.79634607,0.18866773){\color[rgb]{0,0,0}\makebox(0,0)[lb]{\smash{$R_0$}}}%
    \put(0.71420181,0.31470592){\color[rgb]{0,0,0}\makebox(0,0)[lb]{\smash{$R_1$}}}%
    \put(0.90214965,0.308262){\color[rgb]{0,0,0}\makebox(0,0)[lb]{\smash{$S_0$}}}%
    \put(0.49242335,0.17293955){\color[rgb]{0,0,0}\makebox(0,0)[lb]{\smash{$\varphi$}}}%
  \end{picture}%
\endgroup%